\documentclass[10pt]{amsart}
\usepackage{amsfonts}
\usepackage{geometry}
\usepackage{amssymb,amsmath, amsthm,mathtools}
\usepackage[numbers]{natbib}
\usepackage{graphicx}
\usepackage{braket}
\usepackage[pdftex,plainpages=false,colorlinks,hyperindex,bookmarksopen,linkcolor=red,citecolor=blue,urlcolor=blue]{hyperref}
\usepackage{mathrsfs}
\usepackage{xcolor}
\usepackage{epstopdf}
\usepackage{braket}
\usepackage{color}

\pdfoutput=1
\bibpunct{[}{]}{;}{n}{,}{,}
\newtheorem{theorem}{Theorem}[section]

\newtheorem{definition}{Definition}[section]
\theoremstyle{remark}
\newtheorem{remark}{Remark}[section]

\newcommand{\be}{\begin{eqnarray}}
\newcommand{\ee}{\end{eqnarray}}

\newcommand{\func}[1]{\operatorname{#1}}

\numberwithin{equation}{section}
\allowdisplaybreaks

\begin{document}
\title{Renewal processes linked to fractional relaxation equations with
variable order }
\keywords{Fractional relaxation equation, renewal processes, Scarpi
derivative, General fractional calculus, Sonine pair}
\date{\today }

\begin{abstract}
We introduce and study here a renewal process defined by means of
a time-fractional relaxation equation with derivative order
$\alpha(t)$ varying with time $t\geq0$. In particular, we use the
operator introduced by Scarpi in the Seventies (see \cite{SCA})
and later reformulated in the regularized Caputo sense in
\cite{GAR}, inside the framework of the so-called general
fractional calculus. The model obtained extends the well-known
time-fractional Poisson process of fixed order $\alpha \in (0,1)$
and tries to overcome its limitation consisting in the constancy
of the derivative order (and therefore of the memory degree of the
interarrival times) with respect to time. The variable order
renewal process is proved to fall outside the usual subordinated
representation, since it can not be simply defined as a Poisson
process with random time (as happens in the standard fractional
case). Finally a related continuous-time random walk model is
analysed and its limiting behavior established.
\end{abstract}

\author{}
\author{Luisa Beghin$^1$}
\author{Lorenzo Cristofaro$^2$}
\address{${}^1$ Department of Statistical Sciences, Sapienza, University of
Rome. P.le Aldo Moro, 5, Rome, Italy}
\email{luisa.beghin@uniroma1.it}
\address{${}^2$ Department of Statistical Sciences, Sapienza, University of
Rome. P.le Aldo Moro, 5, Rome, Italy}
\email{lorenzo.cristofaro@uniroma1.it}
\author{Roberto Garrappa$^3$}
\address{${}^3$ Department of Mathematics, University of Bari "Aldo Moro",
via Edoardo Orabona 4, Bari, Italy}
\email{roberto.garrappa@uniba.it}
\maketitle

\section{Introduction}

The Poisson process and, in general, the renewal processes are extensively
studied and applied in many different fields, ranging from physics to
finance and actuarial sciences. In particular, their fractional extensions
have been proved to be useful since they are characterized by
non-exponentially distributed intervals between subsequent renewal times. It
is indeed well-known that the time-fractional Poisson process (of order $%
\alpha \in (0,1]$) is a renewal process with interarrival times
following a Mittag-Leffler distribution (with parameter $\alpha $)
(see, for example, \cite{BEG}, \cite{MAI}, \cite{MEE}). The latter
entails a withdrawal from the memoryless property, which is
greater the further away $\alpha $ is from $1$. Although this
model is much more flexible, and adaptable to real data, than the
standard one, there is still a rigidity since the derivative order
(and therefore the memory degree of the intertimes) is constantly
equal to a fixed value $\alpha $ over time.

We introduce and study here a renewal process defined by means of
a time-fractional relaxation equation with order $\alpha (t)$
varying with time $t>0$. The class of suitable functions $\alpha (\cdot )$ is
characterized and some explanatory examples of choices are given; in particular, $\alpha (\cdot )$ can be
modelled to represent two
different variable-order processes: a transition from an initial order $%
\alpha _{1}$ to a second order $\alpha _{2}$ (to be achieved as $%
t\rightarrow +\infty $); a transition from an initial order $\alpha _{1}$ to
a second order $\alpha _{2}$ (to be achieved at a finite time $T$) with a
return the initial value $\alpha _{1}$ as $t\rightarrow +\infty $.
These models can be compared with the renewal processes defined by means of distributed order derivatives (see \cite{BEG2} and \cite{KAT}),
under the assumption of a discrete uniform distribution for the random order $\alpha$ (i.e., taking values $\alpha _{1}$ and $\alpha _{2}$), even if, in our case, the transition
between the two values is depending on the time.

Although different approaches are available in the literature to define
variable-order fractional derivatives, in this work we focus on the operator
introduced by Scarpi in the Seventies (see \cite{SCA}) and later
reformulated in the regularized Caputo sense in \cite{GAR}. The main feature
of this approach is that it formulates a generalization of classic
constant-order operators in the Laplace domain, thus to facilitate the
construction of operators satisfying a Sonine condition.

This work is organized in the following way. In Section
\ref{S:Scarpi} we introduce the variable-order generalization of
the fractional derivative (according to the mentioned approach
introduced by Scarpi) and we recall some basic facts about
time-fractional Poisson processes of constant order. In Section
\ref{S:VO_Renewal} we consider the variable-order fractional
relaxation equation and formulate the basic assumptions needed to
guarantee that its solution is a proper tail distribution for the
interarrival times of a renewal process. In Section \ref{renewal
process} the renewal process defined by means of the previous
results is hence studied and some features, such as the factorial
moments and the autocovariance, are obtained in the Laplace
domain; some graphical representations are provided thanks
to numerical inversion of the corresponding Laplace transformations. Section %
\ref{S:CT_RandomWalk} is devoted to the study of the continuous-time random
walk with counting process represented by the variable-order fractional
renewal and we study its asymptotic behavior, under an appropriate rescaling
and under some assumptions on the jumps distribution.


\section{Preliminaries}

\label{S:Scarpi}

A variable-order fractional derivative can be provided by means of the
following definition (we refer to \cite{GAR} for a more in-depth treatment).

\begin{definition}
\label{def} Let $\alpha :[0,T]\rightarrow (0,1),$ $T\in \mathbb{R}^{+}$, be
a locally integrable function with Laplace transform $A(s):=\int_{0}^{+%
\infty }e^{-st}\alpha (t)dt$ and let $\phi _{A}(t),$ $t\in \lbrack 0,T],$ be
the inverse Laplace transform of $\widetilde{\phi }_{A}(s):=s^{sA(s)-1},$
for $s>0$. For $f\in AC[0,T]$ the (Caputo-type) fractional derivative with
variable order $\alpha (t)$ is defined as%
\begin{equation}
D_{t}^{\alpha (t)}f(t):=\int_{0}^{t}\phi _{A}(t-\tau )f^{\prime
}(\tau )d\tau ,\qquad t\in \lbrack 0,T].  \label{dt}
\end{equation}
\end{definition}

It is easy to check that, for $\alpha (t)=\alpha $ for any $t,$ the operator
$D_{t}^{\alpha (t)}$ coincides with the standard Caputo fractional
derivative of order $\alpha $, since, in this case, $A(s)=\alpha /s$ and $%
\widetilde{\phi }_{A}(s)=s^{\alpha -1}.$ Therefore the kernel is $\phi
_{\alpha }(t)=t^{-\alpha }/\Gamma (1-\alpha )$ and (\ref{dt}) reduces to%
\begin{equation*}
^{C}D_{t}^{\alpha }f(t):=\frac{1}{\Gamma (1-\alpha )}\int_{0}^{t}(t-\tau
)^{-\alpha }f^{\prime }(\tau )d\tau ,\qquad t\in \lbrack 0,T],\;\alpha \in
(0,1).
\end{equation*}

We recall that the Laplace transform (hereafter LT) of $D_{t}^{\alpha (t)}$ is equal to%
\begin{equation}
\mathcal{L}\{D_{t}^{\alpha (t)}u;s\}=s^{sA(s)}\widetilde{u}%
(s)-s^{sA(s)-1}u(0),\qquad s>0,  \label{dt2}
\end{equation}%
where $\mathcal{L}\{u;s\}:=\widetilde{u}(s)=\int_{0}^{+\infty }e^{-sz}u(z)dz$
(see \cite{GAR}).\newline
The operator (\ref{dt}) was analyzed in the framework of the so-called
General Fractional Calculus (see \cite{KOC}, \cite{KOC2}, \cite{KOC3}, \cite%
{LUC}): in particular, it was proved in \cite{GAR} that $D_{t}^{\alpha (t)}$
is invertible under the following assumption%
\begin{equation*}
\lim_{s\rightarrow +\infty }sA(s)=\overline{\alpha }\in (0,1),
\end{equation*}%
which is verified if
\begin{equation}
\lim_{t\rightarrow 0^{+}}\alpha (t)=\overline{\alpha }\in (0,1).  \label{dt3}
\end{equation}%
Then we will assume hereafter that the condition in (\ref{dt3}) is verified;
indeed this is enough to ensure the existence of a real function $\phi
_{A}(\cdot )$ as inverse transform of $\widetilde{\phi }_{A}(s)$.

Moreover, let us denote by $\psi _{A }(\cdot )$ the Sonine pair of $%
\phi _{A }(\cdot )$, i.e. the function such that $\widetilde{\psi }%
_{A}(s)=1/s\widetilde{\phi }_{A}(s)$. Then the inverse operator of $D_{t}^{\alpha (t)}$ is well
defined as%
\begin{equation}
I_{t}^{\alpha (t)}f(t):=\int_{0}^{t}\psi _{A}(t-\tau )f(\tau
)d\tau ,\qquad t\in \lbrack 0,T],  \label{int}
\end{equation}%
for $\psi _{A }(t):=\mathcal{L}^{-1}\{s^{-sA(s)};t\}$, since, thanks to
condition (\ref{dt3}), also the function $\psi _{A }$ is real. It was
proved in \cite{GAR} that the integral in (\ref{int}) enjoys both the
semigroup and symmetry properties and that $\left\{ D_{t}^{\alpha
(t)},I_{t}^{\alpha (t)}\right\} $ satisfies the fundamental theorem of
fractional calculus, i.e. the following holds%
\begin{equation*}
D_{t}^{\alpha (t)}I_{t}^{\alpha (t)}f(t)=f(t),\qquad I_{t}^{\alpha
(t)}D_{t}^{\alpha (t)}f(t)=f(t)-f(0),\qquad t\in \lbrack 0,T].
\end{equation*}%
Finally, the results in \cite{GAR} are obtained for kernels $\widetilde{\phi
}_{A}(\cdot )$ satisfying the following conditions
\begin{subequations}
\begin{eqnarray}
\widetilde{\phi }_{A}(s) &\rightarrow &0,\qquad s\widetilde{\phi }%
_{A}(s)\rightarrow +\infty ,\qquad s\rightarrow +\infty  \label{dt4} \\
\widetilde{\phi }_{A}(s) &\rightarrow &+\infty ,\qquad s\widetilde{\phi }%
_{A}(s)\rightarrow 0,\qquad s\rightarrow 0,  \label{dt6}
\end{eqnarray}%
which are necessary to include Definition \ref{def} in the framework of the
so-called general fractional calculus (see \cite{KOC}, for details).

It seems to be difficult to find examples of functions $\widetilde{\phi }%
_{A}(s)$ (in addition to the limiting case $s^{\alpha -1})$ satisfying (\ref%
{dt4})-(\ref{dt6}) and such that their inverse transforms are
Stieltjes. These three assumptions would be sufficient to ensure
that the solution to the following relaxation equation with
fractional variable order
\end{subequations}
\begin{equation}
D_{t}^{\alpha (t)}u(t)=-\lambda u(t),\qquad u(0)=1,  \label{dt5}
\end{equation}%
is completely monotone (CM), as happens in the (constant-order) fractional
case.  We recall that a function $f:[0,+\infty )\rightarrow \lbrack 0,+\infty
)$ in $C^{\infty}$ is CM if $\ (-1)^{n}f^{(n)}(x)\geq 0$, for any $x\geq 0,$ $n\in \mathbb{N}%
$ (where $f^{(n)}(x):=d^{n}/dx^{n}f(x)$).  However, we do not need the complete monotonicity of the solution to (\ref%
{dt5}) and we will explore below the consequences of its lack to our analysis.

We recall that when $\alpha (t)=\alpha $, for any $t\geq 0,$ the solution to%
\begin{equation}
D_{t}^{\alpha }u(t)=-\lambda u(t),\qquad u(0)=1.  \label{co}
\end{equation}
coincides with $u_{\alpha }(t)=E_{\alpha }(-\lambda t^{\alpha }),$ where $%
E_{\alpha }(x):=\sum_{j=0}^{\infty }x^{j}/\Gamma (\alpha j+1)$ is the
one-parameter Mittag-Leffler function.

The so-called time-fractional Poisson
process $N_{\alpha }:=\left\{ N_{\alpha }(t)\right\} _{t\geq 0}$ can be
defined as a renewal process with interarrival times $Z_{\alpha ,j}$, $%
j=1,2,...,$ independent and identically distributed with $P(Z_{\alpha
}>t)=u_{\alpha }(t)$, $t\geq 0,$ i.e. $N_{\alpha }(t):=\sum_{k=1}^{\infty
}1_{T_{k}^{\alpha }\leq t}$, where $T_{k}^{\alpha }:=\sum_{j=1}^{k}Z_{\alpha
,j}$ (see, for example, \cite{MAI}, \cite{BEG}).

It has also been proved in \cite{MEE} that $N_{\alpha }$ is equal in
distribution to a standard Poisson process time-changed by the inverse of an
independent $\alpha $-stable subordinator (we will denote it as $L_{\alpha
}(t),$ $t\geq 0$, and its density function as $l_{\alpha }(x,t),$ $x,t\geq 0$%
). This result is a consequence of the complete monotonicity of the Mittag-Leffler function, and thus of the solution
to (\ref{co}), since, in this case, we have that%
\begin{equation}
u_{\alpha }(t)=\int_{0}^{+\infty }e^{-\lambda z}l_{\alpha }(z,t)dz
\label{ll}
\end{equation}%
(see \cite{GOR}%
). In other words, it follows since the LT of (\ref{ll}),
i.e. $\widetilde{u}_{\alpha }(s)=s^{\alpha -1}/(s^{\alpha }+\lambda )$, is a
Stieltjes function and thus it coincides with the iterated LT
of a spectral density.

Formula (\ref{ll}) shows that, for the fractional Poisson process $N_{\alpha
}$, the tail distribution function of the interarrival times $%
Z_{\alpha }$ satisfies the following relationship:%
\begin{equation}
P(Z_{\alpha }>t)=P(Z>L_{\alpha }(t)),  \label{ll2}
\end{equation}%
where $Z\sim Exp(\lambda )$ is the interarrival time of the standard Poisson
process $N:=\left\{ N(t)\right\} _{t\geq 0}$.\ From (\ref{ll2}), by
considering that
\begin{equation}
\{T_{k}^{\alpha }<t\}=\{N_{\alpha }(t)>k\},  \label{nt}
\end{equation}%
we have the following equality in the finite-dimensional distributions' sense%
\begin{equation}
N_{\alpha }(t)\overset{f.d.d.}{=}N(L_{\alpha }(t)),  \label{rt}
\end{equation}%
where $L_{\alpha }(t)$ is assumed to be independent of $N(t).$

As we will see below, in the variable order case considered here, a
subordinated representation of the process (analogue to (\ref{rt})) does not
hold, providing an interesting example where the usual correspondence
between time-fractional equations and random time processes does not apply.

\section{The variable-order fractional relaxation equation}

\label{S:VO_Renewal}

Let us consider the solution to the fractional relaxation equation with
variable order derivative (\ref{dt5}). By taking into account (\ref{dt2}),
it is easy to see that its LT reads%
\begin{equation}
\widetilde{u}_{A}(s)=\frac{s^{sA(s)-1}}{\lambda +s^{sA(s)}},\qquad s>0.
\label{uu}
\end{equation}%
In view of what follows, we prove that, under appropriate conditions on $%
\alpha (\cdot )$, the function (\ref{uu}) can be expressed as the
Laplace transform of a tail distribution function, i.e. its
inverse can be written as $u_{A}(t)=P(Z_{A}>t)$, for a positive
r.v. $Z_{A}$.

We recall that a function $g:(0,+\infty)\rightarrow \mathbb{R}$ is
Bernstein if it is $C^{\infty}$, $g(x)\geq 0$, for any $x$, and
$(-1)^{n-1}g^{n}(x) \geq0,$ for any $n \in \mathbb{N},$ $x>0$ (see
\cite{SCH}, p.21).

\begin{theorem}
\label{thm1} Let $\alpha :[0,T]\rightarrow (0,1),$ $T\in \mathbb{R}^{+}$, be
such that the following conditions hold%
\begin{equation}
\lim_{t\rightarrow 0^{+}}\alpha (t)=\alpha ^{\prime },\qquad
\lim_{t\rightarrow +\infty }\alpha (t)=\alpha ^{\prime \prime },  \label{uu2}
\end{equation}%
for $\alpha ^{\prime },\alpha ^{\prime \prime }\in (0,1),$ and that, for its
LT $A(s)$ the function $s^{sA(s)},$ $s>0$, is Bernstein. Then
the solution $u_{A}(t)$ to the relaxation equation (\ref{dt5}) is
non-negative, non-increasing, right-continuous and such that $%
\lim_{t\rightarrow 0^{+}}u_{A}(t)=1$.
\end{theorem}

\begin{proof}
It is easy to check that, if (\ref{uu2}) holds, the conditions (\ref{dt4})-(%
\ref{dt6}) are satisfied, by applying the initial and final value theorems,
respectively (see \cite{LEP}, p.373). Indeed, we have that%
\begin{equation}
\lim_{s\rightarrow +\infty }sA(s)=\alpha ^{\prime },\qquad
\lim_{s\rightarrow 0^{+}}sA(s)=\alpha ^{\prime \prime }  \label{con}
\end{equation}%
(where $\alpha ^{\prime }$ and $\alpha ^{^{\prime \prime }}$ can
coincide). Let now write $s\widetilde{u}_{A}(s)=g(f(s)),$ where $%
f(s):=s^{sA(s)}$ and $g(x):=x/(\lambda +x).$ It is easy to check that $%
g(\cdot )$ is a Bernstein function, so that, under the assumption on $%
s^{sA(s)}$, also $s\widetilde{u}_{A}(s)$ is Bernstein and $\widetilde{u}%
_{A}(s)$ is completely monotone (by applying Corollary 3.8 in \cite{SCH}).

As a consequence, by the Bernstein theorem, there exists a non-negative,
finite measure $\mu \left( \cdot \right) $ on $[0,+\infty )$ such that $%
\widetilde{u}_{A}(s)=\int_{0}^{+\infty }e^{-st}\mu (dt),$ for any $s.$

In order to prove that the inverse LT of $\widetilde{u}%
_{A}(s) $ is a non-increasing and right continuous function (i.e. monotone
of order $1$), we apply Theorem 10 in \cite{WIL}, p.29: it is enough to
check that $\lim_{s\rightarrow +\infty }\widetilde{u}_{A}(s)=0$, that the $%
\lim_{s\rightarrow 0^{+}}s\widetilde{u}_{A}(s)$ exists and that
the first derivative of $s\widetilde{u}_{A}(s)$ is CM and
summable. The latter holds since $s\widetilde{u}_{A}(s)$ is
Bernstein, while the limiting conditions are satisfied by
(\ref{con}). Thus $\widetilde{u}_{A}(s)$ is the Laplace transform
of a non-negative, non-increasing, right-continuous function,
which coincides with the solution to (\ref{dt5}). Finally, since $\widetilde{%
u}_{A}(s)\sim 1/s,$ for $s\rightarrow +\infty ,$ we can apply the Tauberian
theorem (see \cite{FEL}) in order to check that $\lim_{t\rightarrow
0^{+}}u_{A}(t)=1$.
\end{proof}

\ \bigskip

We now provide some explanatory examples of functions $\alpha (\cdot
)$ for which the previous result holds, in addition to the
constant-order case.
Obviously, when $\alpha (t)=\alpha \in (0,1)$, $\forall t$, we have that $%
s^{sA(s)}=s^{\alpha }$ is a Bernstein function and
\begin{equation*}
\widetilde{u}_{A}(s)=\frac{s^{sA(s)-1}}{\lambda +s^{sA(s)}}=\frac{s^{\alpha
-1}}{\lambda +s^{\alpha }}.
\end{equation*}%
Its inverse LT is the Mittag-Leffler function $u_{\alpha
}(t)=E_{\alpha }(-\lambda t^{\alpha }),$ which is completely monotone for $%
0<\alpha \leq 1$ (see \cite{GOR} and \cite{SCN}).

\subsection{Exponential transition from $\protect\alpha_1$ to $\protect\alpha%
_2$}

A special case is obtained by means of the function
\begin{equation*}
\alpha (t)=\alpha _{1}+(\alpha _{1}-\alpha _{2})\mathrm{e}^{-ct},\quad
\alpha _{1},\alpha _{2}\in (0,1),\quad c>0,
\end{equation*}%
describing the order transition from $\alpha _{1}$ to $\alpha _{2}$
according to an exponential law with rate $-c$ \cite{GAR}. It is immediate
to compute its LT, $A(s)$, and the corresponding function $s^{sA(s)}$, as
\begin{equation*}
A(s)=\frac{\alpha _{2}c+\alpha _{1}s}{s(c+s)},\quad s^{sA(s)}=s^{\frac{%
\alpha _{2}c+\alpha _{1}s}{c+s}}.
\end{equation*}

Finding all possible choices of parameters $\alpha _{1}$, $\alpha _{2}$ and $%
c$ in order to guarantee that $s^{sA(s)}$ is Bernstein remains an open
problem. Numerical inversion of the LT (according to the procedure outlined
in \cite{GAR}) allows however to observe the existence of some sets of
parameters for which the solution to the renewal equation (\ref{dt5})
displays the properties ensured by Theorem \ref{thm1}. Indeed, as we show in
Figure \ref{fig:Fig_RelEq_Exp_DerSol}, for the considered sets of
parameters, we obtain non-negative solutions of the relaxation equation
(left plot) which are also non-increasing, as one can argue by observing the
non-positive character of their first-order derivatives (right plot).

\begin{figure}[htb]
\centering
\begin{tabular}{c@{\hspace{1.0cm}}c}
\includegraphics[width=0.45\textwidth]{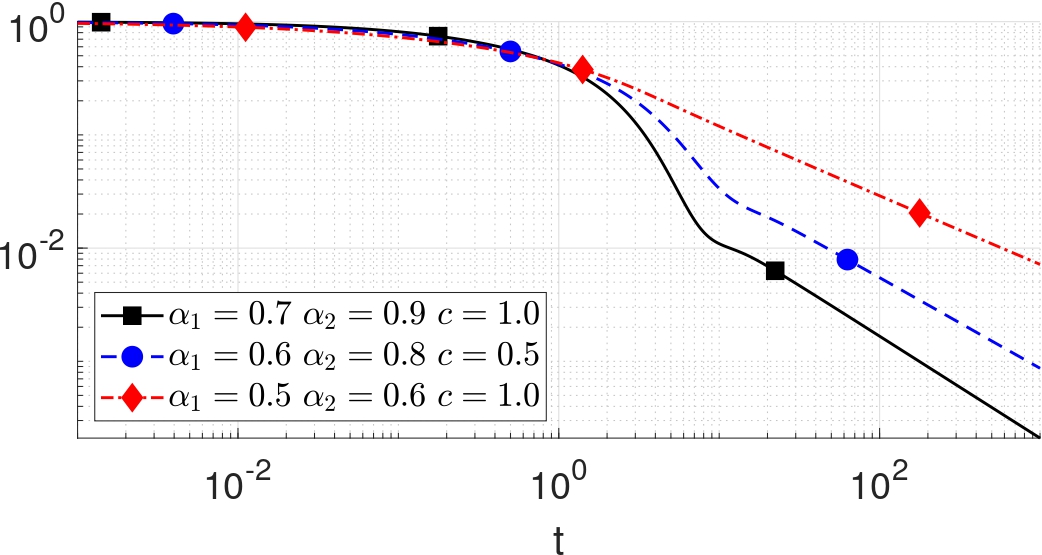} & %
\includegraphics[width=0.45\textwidth]{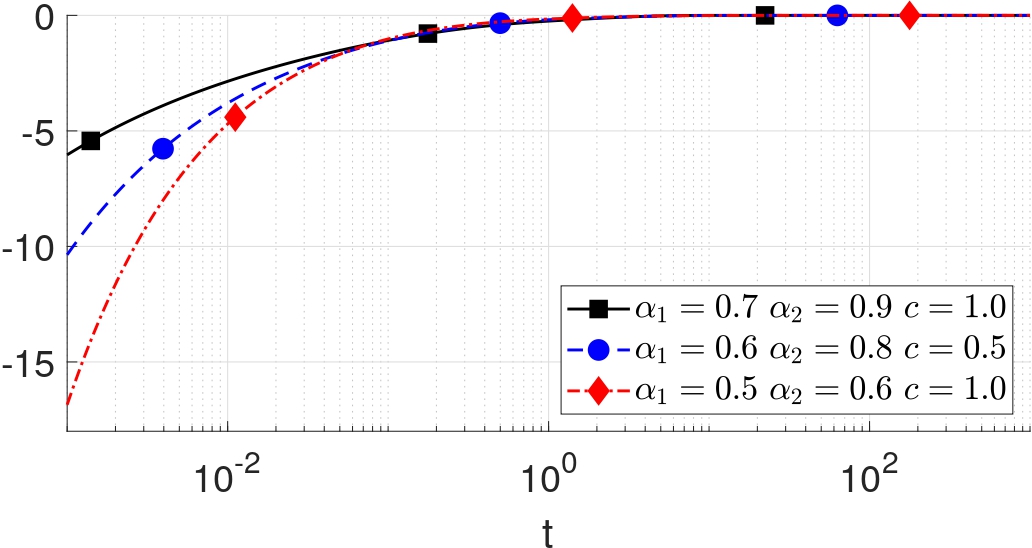}%
\end{tabular}%
\caption{Solution $u_A(t)$ (left plot), and its first-order derivative $%
u^{\prime}_A(t)$ (right plot), of the variable-order relaxation equation
with $\protect\alpha (t)=\protect\alpha _{1}+(\protect\alpha _{1}-\protect%
\alpha _{2})\mathrm{e}^{-ct}$ and different parameters $\protect\alpha_1$, $%
\protect\alpha_2$ and $c$.}
\label{fig:Fig_RelEq_Exp_DerSol}
\end{figure}

\subsection{Exponential transition with return}

A further transition, recently introduced in \cite{GAR3}, is obtained by
means of the function
\begin{equation}  \label{eq:TransitionPeak}
\alpha(t) = \alpha_1 + (\alpha_2-\alpha_1) \frac{\mathrm{e}^{-c_1 t} -
\mathrm{e}^{-c_2 t}}{F_c(c_2-c_1)} , \quad \alpha_1,\alpha_2 \in (0,1),
\quad c_1,c_2 > 0.
\end{equation}

Unlike the previous one, this function describes an order transition which
starts from $\alpha _{1}$, increases (or decreases) to $\alpha _{2}$ and
hence returns back to $\alpha _{1}$ as $t\rightarrow \infty $. Thus, in this case, the condition (\ref%
{uu2}) holds for $\alpha ^{\prime }=\alpha ^{\prime \prime }=\alpha _{1}.$
The constant $F_{c}$ is chosen so that $\alpha (t)$ has maximum or minimum
value $\alpha _{2}$, and hence it is given by
\begin{equation*}
F_{c}=\frac{1}{c_{2}-c_{1}}\left[ \Bigl(\frac{c_{1}}{c_{2}}\Bigr)^{\frac{%
c_{1}}{c_{2}-c_{1}}}-\Bigl(\frac{c_{1}}{c_{2}}\Bigr)^{\frac{c_{2}}{%
c_{2}-c_{1}}}\right] ,
\end{equation*}%
and $\alpha _{2}$ is achieved at time $t=(c_{2}-c_{1})^{-1}\log c_{2}/c_{1}$%
. Moreover, it is simple to evaluate
\begin{equation*}
A(s)=\frac{1}{s}\alpha _{1}+\frac{\alpha _{2}-\alpha _{1}}{%
F_{c}(s+c_{1})(s+c_{2})},\quad s^{sA(s)}=s^{\alpha _{1}}s^{\frac{s(\alpha
_{2}-\alpha _{1})}{F_{c}(s+c_{1})(s+c_{2})}}.
\end{equation*}

Also in this case a precise characterization of the whole set of
possible choices for $\alpha_1$, $\alpha_2$, $c_1$ and $c_2$ to
ensure that $s^{sA(s)} $ is Bernstein does not seem possible.
Again, numerical inversion of the LT is used to guarantee that
there exist some sets of parameters such that the solution to the
renewal equation (\ref{dt5}) has the properties required in
Theorem \ref{thm1}. From Figure \ref{fig:Fig_RelEq_Peak_DerSol} we
observe the non-negativity of these solutions (left plot) and its
non-increasing character expressed as non-positivity of the
corresponding first-order derivatives (right plot).

\begin{figure}[htb]
\centering
\begin{tabular}{c@{\hspace{1.0cm}}c}
\includegraphics[width=0.45\textwidth]{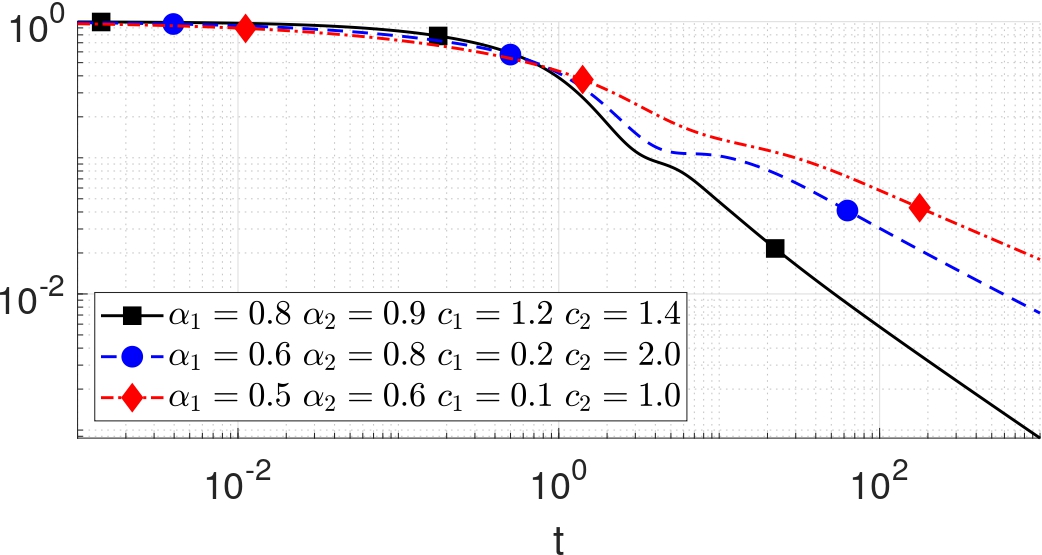} & %
\includegraphics[width=0.45\textwidth]{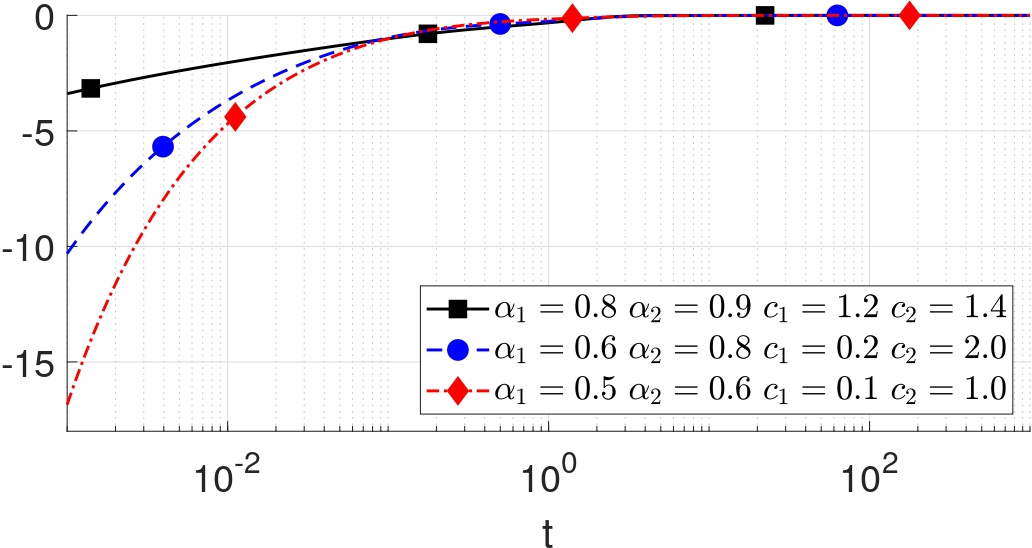}%
\end{tabular}%
\caption{Solution $u_A(t)$ (left plot), and its first-order derivative $%
u^{\prime}_A(t)$ (right plot), of the variable-order relaxation equation
with $\protect\alpha(t) = \protect\alpha_1 + (\protect\alpha_2-\protect\alpha%
_1) \frac{\mathrm{e}^{-c_1 t} - \mathrm{e}^{-c_2 t}}{F_c(c_2-c_1)}$ and
different parameters $\protect\alpha_1$, $\protect\alpha_2$, $c_1$ and $c_2$.
}
\label{fig:Fig_RelEq_Peak_DerSol}
\end{figure}

\section{The variable-order fractional renewal process}

\label{renewal process}

By resorting to the results obtained so far, we can define a renewal process
by assuming that its interarrival times have tail distribution function
equal to the solution of the relaxation equation (\ref{dt5}).

\begin{definition}
\label{def2}Let $N_{A}(t):=\left\{ N_{A}(t)\right\} _{t\geq 0}$ be a renewal process
with interarrival times $Z_{A,j}$, $j=1,2,...,$ independent and identically
distributed with $P(Z_{A}>t)=u_{A}(t)$, where $u_{A}(t)$, $t\geq 0,$
coincides with the solution of (\ref{dt5}).
\end{definition}

The density function of $Z_{A,j}$ can be written in Laplace domain as%
\begin{equation}
\widetilde{f}_{Z_{A}}(s\}=\frac{\lambda }{\lambda +s^{sA(s)}},  \label{ff}
\end{equation}%
while the LT of the $k$-th renewal time density reads%
\begin{equation}
\widetilde{f}_{T_{k}^{A}}(s)=\frac{\lambda ^{k}}{\left( \lambda
+s^{sA(s)}\right) ^{k}},\qquad k=1,2,...,  \label{tk}
\end{equation}%
where $T_{k}^{A}:=\sum_{j=1}^{k}Z_{A,j}.$ Thus the probability mass function
(in Laplace domain) of $N_{A}$ can be obtained as follows%
\begin{eqnarray}
\widetilde{p}_{k}^{A}(s) &:=& \mathcal{L}\left\{ P\left(
N_{A}(t)=k\right) ;s\right\} =\frac{\lambda ^{k}}{s\left( \lambda
+s^{sA(s)}\right) ^{k}}-\frac{\lambda ^{k+1}}{s\left( \lambda
+s^{sA(s)}\right) ^{k+1}}  \label{pk} \\
&=&\frac{\lambda ^{k}s^{sA(s)-1}}{\left( \lambda +s^{sA(s)}\right) ^{k+1}}%
,\qquad k=0,1,...,\;t\geq 0,  \notag
\end{eqnarray}%
and $p_{k}^{A}(t)$ satisfies the following Cauchy problem%
\begin{equation}
D_{t}^{\alpha (t)}p_{k}(t)=-\lambda (p_{k}(t)-p_{k-1}(t)),\qquad
p_{k}(0)=1_{\{0\}}(k),  \label{eq}
\end{equation}%
for $k=0,1,2,...$ and$\;t\geq 0.$

It is proved in \cite{GAR}, by some counterexamples, that, in the variable
order case, $\widetilde{\phi }_{A}(s)$ is not in general a Stieltjes
function; as a consequence, also the function (\ref{uu}) is not Stieltjes.
Thus, in our case, the solution of the relaxation equations $u_A (t)$ can
not be expressed as integral of the exponential tail distribution (as in (%
\ref{ll})) and a time-change representation (analogue to that given in (\ref%
{rt})) does not hold for the renewal process $N_{A}.$

We give in Figure \ref{fig:Fig_PMF_Exp} the probability mass function $p^{A}_k (t)$, for small values of $k$,
 in the first explanatory special case introduced above (i.e. for $\protect\alpha (t)=\protect\alpha _{1}+(%
\protect\alpha _{1}-\protect\alpha _{2})\mathrm{e}^{-ct}$). One can observe that, with the
exponential transition from $\alpha_1$ to $\alpha_2$, the variable-order
probability mass functions have a similar behavior to the corresponding
functions of order $\alpha_1$ for $t\to 0^+$ and of order $\alpha_2$ as $t
\to \infty$.

\begin{figure}[htb]
\centering
\begin{tabular}{c@{\hspace{1.0cm}}c}
\includegraphics[width=0.45\textwidth]{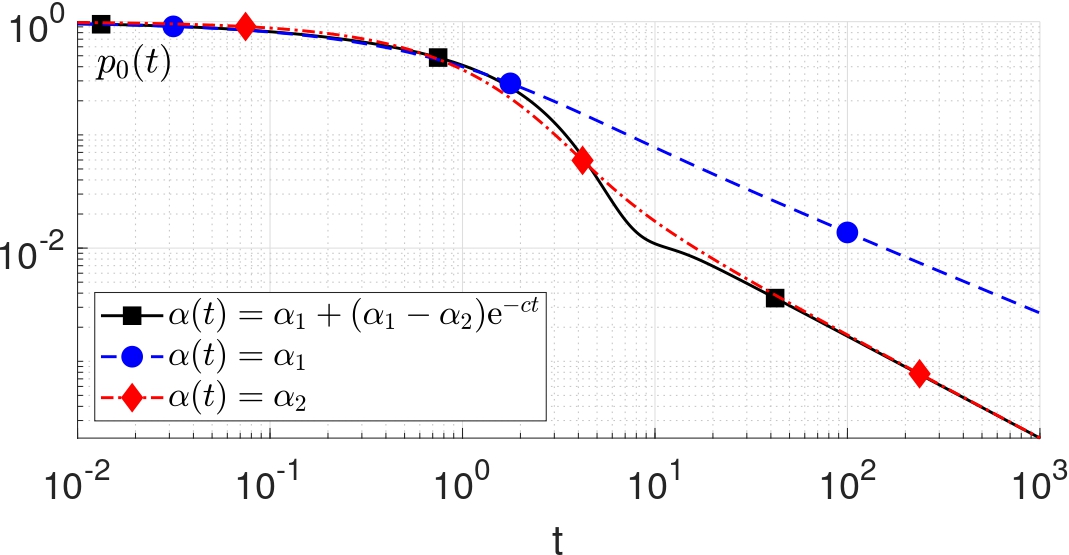}
&
\includegraphics[width=0.45\textwidth]{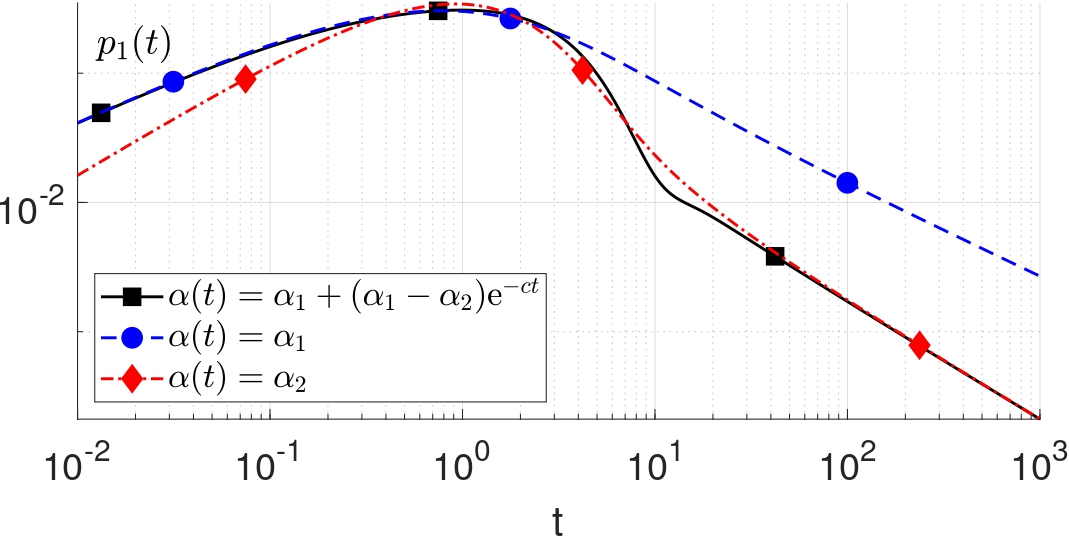}
\\
\includegraphics[width=0.45\textwidth]{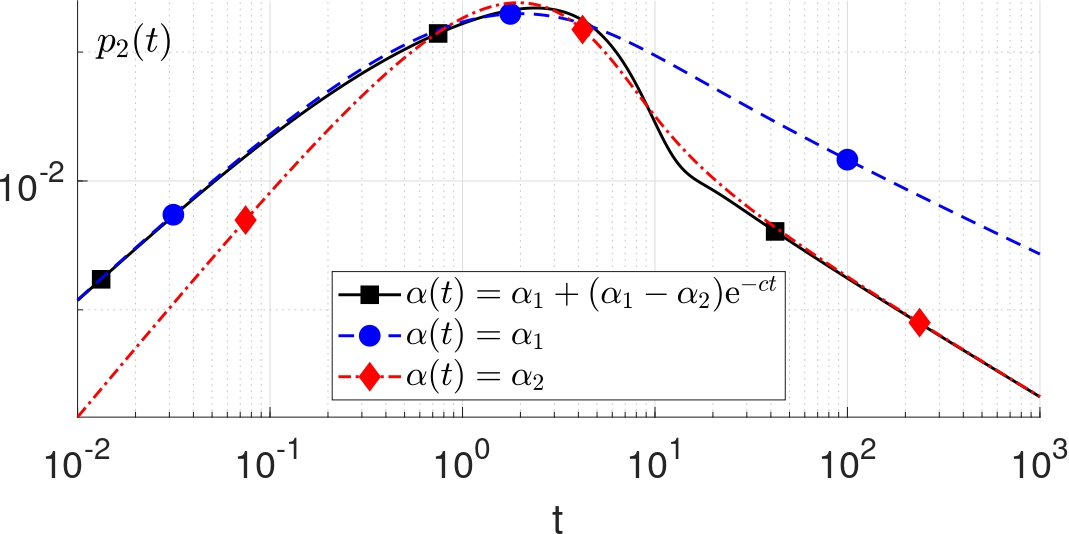}
&
\includegraphics[width=0.45\textwidth]{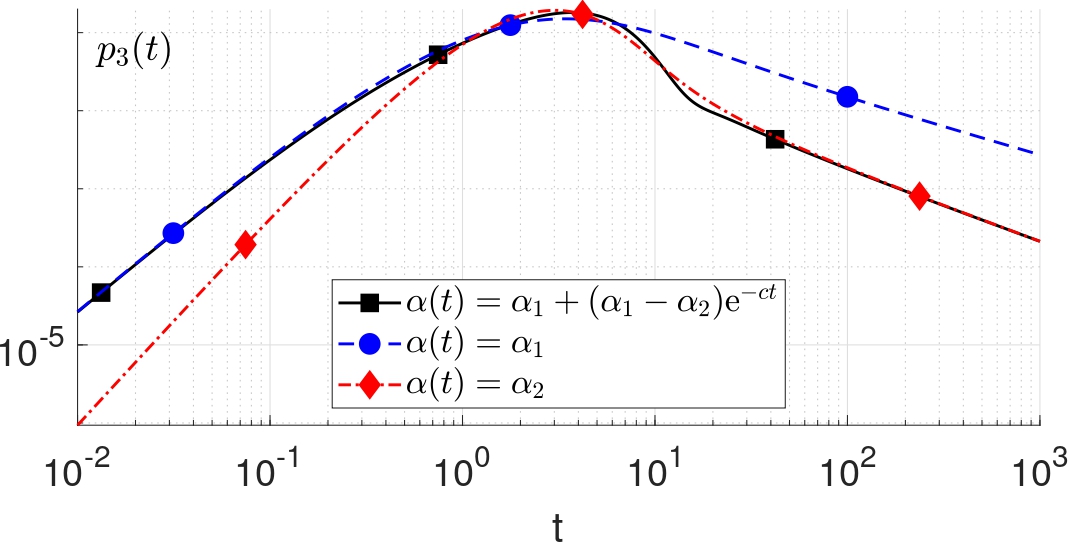}
\\
&
\end{tabular}%
\caption{Comparison of probability mass functions $p^{A}_k(t)$, $k=0,1,2,3$
between exponential variable-order $\protect\alpha (t)=\protect\alpha _{1}+(%
\protect\alpha _{1}-\protect\alpha _{2})\mathrm{e}^{-ct}$ and constant
orders $\protect\alpha_1$ and $\protect\alpha_2$ (here $\protect\alpha_1=0.7$%
, $\protect\alpha_2=0.9$ and $c=1.0$).}
\label{fig:Fig_PMF_Exp}
\end{figure}

On the other side, as one can observe from Figure \ref{fig:Fig_PMF_Peak},
with the variable-order transition (\ref{eq:TransitionPeak}), the behavior
is similar to the behavior of the probability mass functions of constant
order $\alpha_1$ both as $t\to 0^+$ and as $t \to \infty$, while the
behavior with the constant order $\alpha_2$ is replicated just on short
intervals at medium times.

\begin{figure}[htb]
\centering
\begin{tabular}{c@{\hspace{1.0cm}}c}
\includegraphics[width=0.45%
\textwidth]{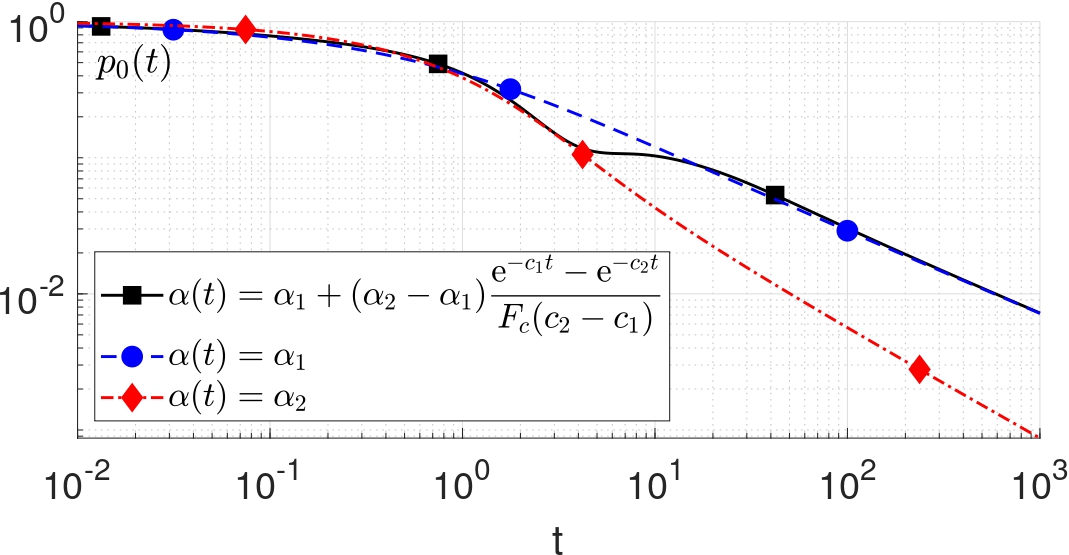} & %
\includegraphics[width=0.45%
\textwidth]{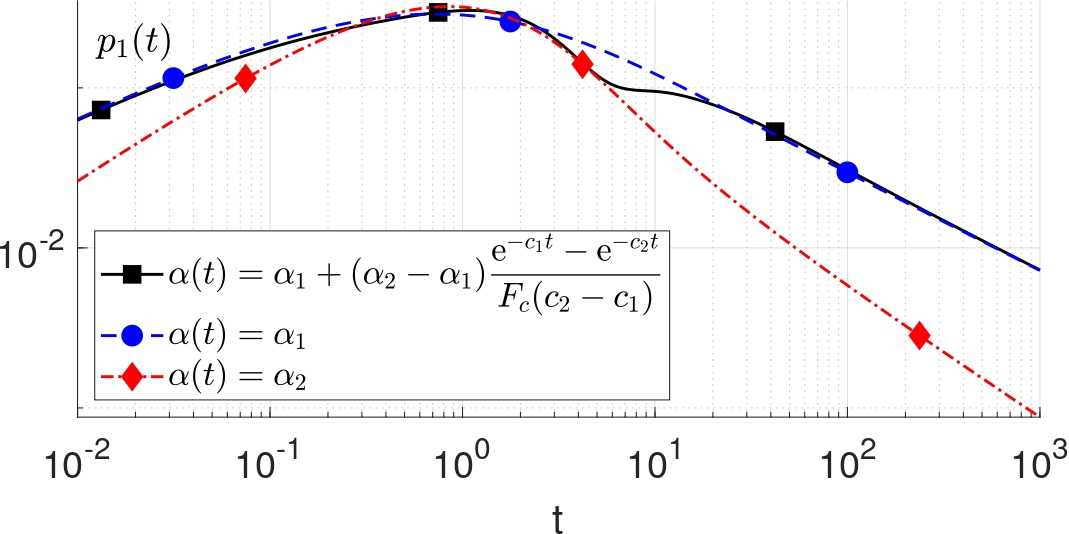} \\
\includegraphics[width=0.45%
\textwidth]{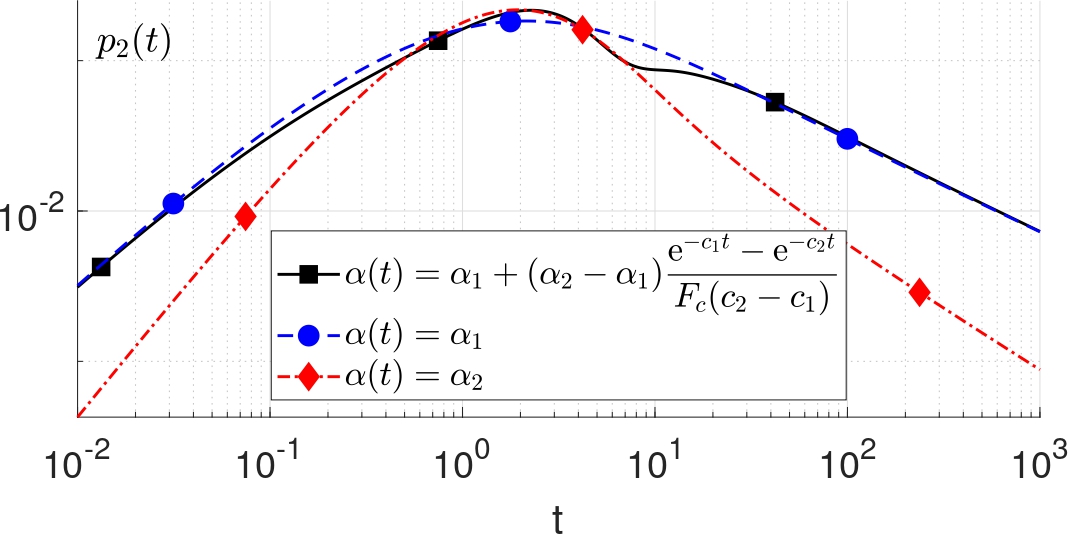} & %
\includegraphics[width=0.45%
\textwidth]{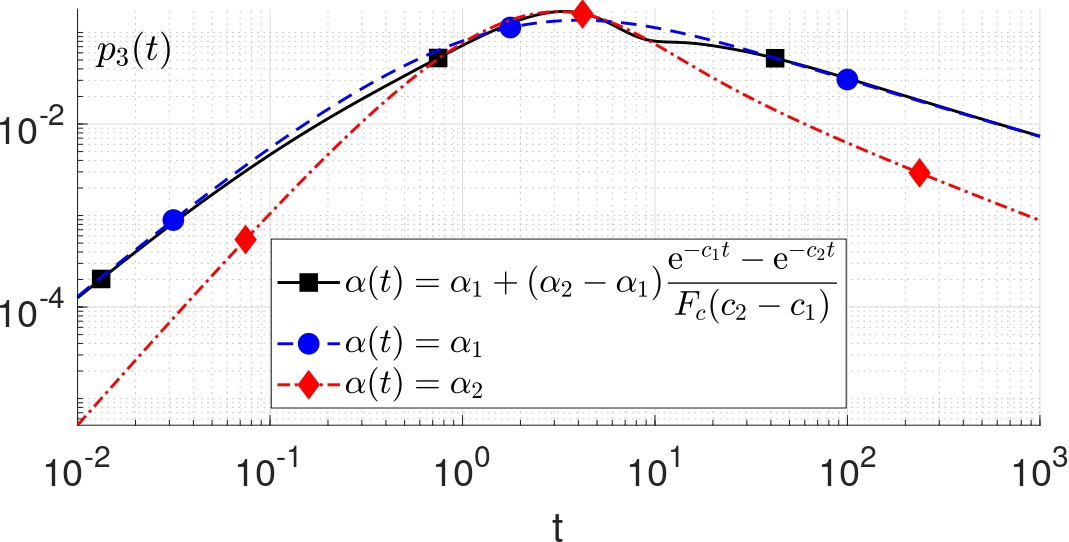} \\
&
\end{tabular}%
\caption{Comparison of probability mass functions $p^{A}_k(t)$, $k=0,1,2,3$
between exponential variable-order $\protect\alpha(t) = \protect\alpha_1 + (%
\protect\alpha_2-\protect\alpha_1) \frac{\mathrm{e}^{-c_1 t} - \mathrm{e}%
^{-c_2 t}}{F_c(c_2-c_1)}$ and constant order $\protect\alpha_1$ (here $%
\protect\alpha_1=0.6$, $\protect\alpha_2=0.8$, $c_1=0.2$ and $c_2=2.0$).}
\label{fig:Fig_PMF_Peak}
\end{figure}

We are now interested in the properties of the above defined process,
starting from its factorial moments and the moments of its interarrival
times.

\begin{theorem}
The $r$-th factorial moment of $N_{A}$, $r\in \mathbb{N}$, has
LT
\begin{equation}
\mathcal{L}\left\{ \mathbb{E}\left[ N_{A}(t)\cdot \cdot \cdot (N_{A}(t)-r+1)%
\right] ;s\right\} =\frac{r!\lambda ^{r}}{s^{rsA(s)+1}}.  \label{fac}
\end{equation}%
Moreover, the $r$-th moment of its interarrival time $Z_{A}$ is infinite for
any $r\in \mathbb{N}$.
\end{theorem}

\begin{proof}
In order to prove formula (\ref{fac}) we derive the expression of
the probability generating function of $N_{A}$ (in the Laplace
domain), as follows, for $|u|<1$,

\begin{eqnarray}
\widetilde{G}_{N_{A}}(u;s) &:=& \mathcal{L}\left\{
G_{N_{A}}(u;t);s\right\} =\sum_{k=0}^{\infty
}u^{k}\widetilde{p}^{A}_{k}(s)
\label{gu} \\
&=&[\text{by (\ref{pk})}]  \notag \\
&=&\frac{s^{sA(s)-1}}{\lambda +s^{sA(s)}}\sum_{k=0}^{\infty }\frac{(u\lambda
)^{k}}{\left( \lambda +s^{sA(s)}\right) ^{k}}  \notag \\
&=&\frac{s^{sA(s)-1}}{\lambda (1-u)+s^{sA(s)}},  \notag
\end{eqnarray}%
Now, by taking the $r$-th order derivative of (\ref{gu}), for $%
u=1$, formula (\ref{fac}) easily follows.

As far as the moments of the interarrival times are concerned, we first
prove that the expected value is infinite: indeed we have that%
\begin{eqnarray*}
\mathbb{E}Z_{A} &=&\lim_{s\rightarrow 0^{+}}\int_{0}^{+\infty
}e^{-st}P(Z_{A}>t)dt \\
&=&\lim_{s\rightarrow 0^{+}}\frac{s^{sA(s)-1}}{\lambda +s^{sA(s)}}=+\infty ,
\end{eqnarray*}%
where the interchange between limit and integral is justified by the
monotone convergence theorem. The last step follows by applying the
conditions (\ref{uu2}), which imply (\ref{con}), and by considering that $%
\alpha',\alpha'' \in (0,1),$ so that $\lim_{s\rightarrow
0^{+}}s^{sA(s)-1}=+\infty $ and $\lim_{s\rightarrow 0^{+}}s^{sA(s)}=0$.
Finally, by applying the Holder's inequality to $Z_{A}$ and taking into
account that it is a non-negative random variable, we can conclude that the
moments are infinite for any $r=2,3,..$
\end{proof}

In order to evaluate the autocovariance of $N_{A}$ (at least in the
Laplace domain)$,$ we recall the following result by \cite{SUY}, which holds
for any renewal process $M(t):=\left\{ M(t)\right\} _{t\geq 0}$ with density function of the
interarrival times $f(\cdot )$:
\begin{equation}
\int_{0}^{+\infty }\int_{0}^{+\infty }e^{-s_{1}t_{1}-s_{2}t_{2}}\mathbb{E}%
M(t_{1})M(t_{2})dt_{1}dt_{2}=\frac{\left[ 1-\widetilde{f}(s_{1})\widetilde{f}%
(s_{2})\right] \widetilde{f}(s_{1}+s_{2})}{s_{1}s_{2}\left[ 1-\widetilde{f}%
(s_{1})\right] \left[ 1-\widetilde{f}(s_{2})\right] \left[ 1-\widetilde{f}%
(s_{1}+s_{2})\right] },  \label{su}
\end{equation}%
for $s_{1},s_{2}\geq 0.$ By considering (\ref{ff}), we immediately obtain
from (\ref{su}) that%
\begin{eqnarray}
&&\int_{0}^{+\infty }\int_{0}^{+\infty }e^{-s_{1}t_{1}-s_{2}t_{2}}Cov\left[
N_{A}(t_{1}),N_{A}(t_{2})\right] dt_{1}dt_{2}  \label{cov} \\
&=&\frac{\lambda ^{2}\left[
s_{1}^{s_{1}A(s_{1})}+s_{2}^{s_{2}A(s_{2})}-(s_{1}+s_{2})^{(s_{1}+s_{2})A((s_{1}+s_{2}))}%
\right] +\lambda s_{1}^{s_{1}A(s_{1})}s_{2}^{s_{2}A(s_{2})}}{%
s_{1}^{s_{1}A(s_{1})+1}s_{2}^{s_{2}A(s_{2})+1}(s_{1}+s_{2})^{(s_{1}+s_{2})A((s_{1}+s_{2}))}.%
}  \notag
\end{eqnarray}

It is possible to check that, in the fixed order case, i.e. for $%
sA(s)=\alpha ,$ formula (\ref{cov}) reduces to the LT of the
well-known autocovariance of the fractional Poisson process, which is equal
to:%
\begin{equation}
Cov\left[ N_{\alpha }(t_{1}),N_{\alpha }(t_{2})\right] =\frac{\lambda \left(
t_{1}\wedge t_{2}\right) ^{\alpha }}{\Gamma (1+\alpha )}+\frac{\lambda ^{2}}{%
\Gamma (1+\alpha )^{2}}\left[ \alpha \left( t_{1}\wedge t_{2}\right)
^{2\alpha }B(\alpha ,\alpha +1)+F(\alpha ;t_{1}\wedge t_{2};t_{1}\vee t_{2})%
\right] ,  \label{cov1}
\end{equation}%
where $B(\alpha ,\beta ):=\int_{0}^{1}x^{\alpha -1}(1-x)^{\beta -1}dx$ is
the Beta function, $\alpha ,\beta \geq 0$, $F(\alpha ;x;y):=\alpha
y^{2\alpha }B(\alpha ,\alpha +1;x/y)-x^{\alpha }y^{\alpha }$ and $B(\alpha
,\beta ;x):=\int_{0}^{x}y^{\alpha -1}(1-y)^{\beta -1}dy$ is the incomplete
Beta function, for $x\in (0,1]$, $\alpha ,\beta \geq 0$ (see \cite{LEO}). By
taking the double LT of (\ref{cov1}) we have that%
\begin{eqnarray*}
&&\int_{0}^{+\infty }\int_{0}^{+\infty }e^{-s_{1}t_{1}-s_{2}t_{2}}Cov\left[
N_{\alpha }(t_{1}),N_{\alpha }(t_{2})\right] dt_{1}dt_{2} \\
&=&\frac{\lambda }{\Gamma (1+\alpha )}\int_{0}^{+\infty }e^{-s_{2}t_{2}}%
\left[ \int_{0}^{t_{2}}e^{-s_{1}t_{1}}t_{1}^{\alpha
}dt_{1}+\int_{t_{2}}^{+\infty }e^{-s_{1}t_{1}}t_{1}^{\alpha }dt_{1}\right]
dt_{2}+ \\
&&+\frac{\lambda ^{2}}{\Gamma (1+2\alpha )}\int_{0}^{+\infty }e^{-s_{2}t_{2}}%
\left[ \int_{0}^{t_{2}}e^{-s_{1}t_{1}}t_{1}^{2\alpha }dt_{1}+t_{2}^{2\alpha
}\int_{t_{2}}^{+\infty }e^{-s_{1}t_{1}}dt_{1}\right] dt_{2}+ \\
&&+\frac{\lambda ^{2}\alpha }{\Gamma (1+\alpha )^{2}}\int_{0}^{+\infty
}e^{-s_{2}t_{2}}t_{2}^{2\alpha
}dt_{2}\int_{0}^{t_{2}}e^{-s_{1}t_{1}}dt_{1}\int_{0}^{t_{1}/t_{2}}z^{\alpha
-1}(1-z)^{\alpha }dz+ \\
&&+\frac{\lambda ^{2}\alpha }{\Gamma (1+\alpha )^{2}}\int_{0}^{+\infty
}e^{-s_{2}t_{2}}dt_{2}\int_{t_{2}}^{+\infty }e^{-s_{1}t_{1}}t_{1}^{2\alpha
}dt_{1}\int_{0}^{t_{2}/t_{1}}z^{\alpha -1}(1-z)^{\alpha }dz+ \\
&&-\frac{\lambda ^{2}}{\Gamma (1+\alpha )^{2}}\int_{0}^{+\infty
}e^{-s_{2}t_{2}}t_{2}^{\alpha }dt_{2}\int_{0}^{+\infty
}e^{-s_{1}t_{1}}t_{1}^{\alpha }dt_{1} \\
&=:&
I_{s_{1},s_{2}}^{I}+I_{s_{1},s_{2}}^{II}+I_{s_{1},s_{2}}^{III}+I_{s_{2},s_{1}}^{III}+I_{s_{1},s_{2}}^{IV}.
\end{eqnarray*}%
By some calculations we easily obtain the following results:%
\begin{eqnarray}
I_{s_{1},s_{2}}^{I} &=&\frac{\lambda }{s_{1}s_{2}(s_{1}+s_{2})^{\alpha }}
\label{1} \\
I_{s_{1},s_{2}}^{II} &=&\frac{\lambda ^{2}}{s_{1}s_{2}(s_{1}+s_{2})^{2\alpha
}}  \label{2} \\
I_{s_{1},s_{2}}^{IV} &=&\frac{\lambda ^{2}}{s_{1}^{1+\alpha
}+s_{2}^{1+\alpha }},  \label{3}
\end{eqnarray}%
while for the terms of the third type, we must take into account the
following formula (see (1.6.15) together with (1.6.14) and (1.9.3) in \cite%
{KIL}):%
\begin{equation*}
\int_{0}^{1}e^{zt}t^{a-1}(1-t)^{c-a-1}dt=\Gamma (c-a)E_{1,c}^{a}(z),
\end{equation*}%
for $0<\func{Re}(a)<\func{Re}(c)$, where $E_{\alpha ,\beta
}^{\gamma }\left( \cdot \right) $ is the Mittag-Leffler function
with three parameters (also
called Prabhakar function), for any $x\in \mathbb{C}$,%
\begin{equation*}
E_{\alpha ,\beta }^{\gamma }\left( x\right) :=\sum_{j=0}^{\infty }\frac{%
(\gamma )_{j}x^{j}}{j!\Gamma (\alpha j+\beta )},\qquad \alpha ,\beta ,\gamma
\in \mathbb{C}\text{, }\func{Re}(\alpha )>0,
\end{equation*}%
for $(\gamma )_{j}:=\Gamma (\gamma +j)/\Gamma (\gamma ).$ We also recall the
well-known formula (see \cite{KIL}, p.47)
\begin{equation}
\mathcal{L}\left\{ t^{\beta -1}E_{\alpha ,\beta }^{\gamma }(at^{\alpha
});s\right\} =\frac{s^{\alpha \gamma -\beta }}{\left( s^{\alpha }-a\right)
^{\gamma }},\qquad |as^{-\alpha }|<1.  \label{ml}
\end{equation}%
Thus we can write
\begin{eqnarray}
I_{s_{1},s_{2}}^{III} &=&\frac{\lambda ^{2}\alpha }{\Gamma (1+\alpha )^{2}}%
\int_{0}^{+\infty }e^{-s_{2}t_{2}}t_{2}^{2\alpha
}dt_{2}\int_{0}^{1}z^{\alpha -1}(1-z)^{\alpha
}dz\int_{zt_{2}}^{t_{2}}e^{-s_{1}t_{1}}dt_{1}  \label{4} \\
&=&\frac{\lambda ^{2}\alpha }{\Gamma (1+\alpha )^{2}}\frac{1}{s_{1}}%
\int_{0}^{+\infty }e^{-s_{2}t_{2}}t_{2}^{2\alpha
}dt_{2}\int_{0}^{1}z^{\alpha -1}(1-z)^{\alpha }\left[
e^{-s_{1}t_{2}z}-e^{-s_{1}t_{2}}\right] dz  \notag \\
&=&\frac{\lambda ^{2}}{s_{1}}\left[ \int_{0}^{+\infty
}e^{-s_{2}t_{2}}t_{2}^{2\alpha }E_{1,2\alpha +1}^{\alpha }\left(
-s_{1}t_{2}\right) dt_{2}-\frac{1}{(s_{1}+s_{2})^{2\alpha +1}}\right]  \notag
\\
&=&\frac{\lambda ^{2}}{s_{1}}\left[ \frac{1}{s_{2}^{\alpha
+1}(s_{1}+s_{2})^{\alpha }}-\frac{1}{(s_{1}+s_{2})^{2\alpha +1}}\right]
\notag
\end{eqnarray}%
and, analogously, for $I_{s_{2},s_{1}}^{III}.$ In view of (\ref%
{1}), (\ref{2}), (\ref{3}) and (\ref{4}), we obtain that%
\begin{equation*}
\int_{0}^{+\infty }\int_{0}^{+\infty }e^{-s_{1}t_{1}-s_{2}t_{2}}Cov\left[
N_{\alpha }(t_{1}),N_{\alpha }(t_{2})\right] dt_{1}dt_{2}=\frac{\lambda
s_{1}^{\alpha }s_{2}^{\alpha }+\lambda ^{2}\left[ s_{1}^{\alpha
}+s_{2}^{\alpha }-(s_{1}+s_{2})^{\alpha }\right] }{s_{1}^{\alpha
+1}s_{2}^{\alpha +1}(s_{1}+s_{2})^{\alpha }},
\end{equation*}%
which coincides with (\ref{cov}), when $sA(s)=\alpha .$

\section{The related continuous-time random walk and its limiting process}

\label{S:CT_RandomWalk}

Based on the previous results, we consider the continuous-time random walk
(hereafter CTRW) defined by means of the counting process $N_{A}$: let $%
X_{i},i=1,2,...$ be real, independent random variables with common density
function $f_{X}(\cdot )$ and let us denote $\widehat{g}(\kappa ):=\int_{%
\mathbb{R}}e^{i\kappa x}g(x)dx$, for $\kappa \in \mathbb{R}$ and for a
function $g:\mathbb{R}\rightarrow \mathbb{R}$, for which the integral
converges. We define, for any $t\geq 0,$ the CTRW with driving counting
process $N_{A}$ and jumps $X_{i}$ (under the assumption that $N_{A}$ and $%
X_{i}$ are independent each other) as
\begin{equation}
Y_{A}(t):=\sum_{i=1}^{N_{A}(t)}X_{i},  \label{ya}
\end{equation}%
and denote its density as $f_{Y_{A}}(y,t):=P(Y_{A}(t)\in dy)/dy.$ Then it is
well-known that the LT of the characteristic function of $%
Y_{A}(t)$ reads, for any $t\geq 0,$%
\begin{equation*}
\mathcal{L}\left\{ \widehat{f}_{Y_{A}}(\kappa ,t);s\right\} =\frac{1-%
\widetilde{f}_{Z_{A}}(s)}{s\left[ 1-\widetilde{f}_{Z_{A}}(s)\widehat{f}%
_{X}(\kappa )\right] },\qquad s\geq 0,\text{ }\kappa \in \mathbb{R},
\end{equation*}%
where $\widetilde{f}_{Z_{A}}(s)$ is the LT of the
interarrivals' density. By considering (\ref{ff}), we get%
\begin{equation}
\mathcal{L}\left\{ \widehat{f}_{Y_{A}}(\kappa ,t);s\right\} =\frac{%
s^{sA(s)-1}}{s^{sA(s)}+\lambda \lbrack 1-\widehat{f}_{X}(\kappa )]}.
\label{ya2}
\end{equation}%
We are now able to study the limiting behavior of the CTRW under an
appropriate rescaling. To this aim, we recall the definition of the
time-space fractional diffusion $Y_{\alpha ,\beta }^{\vartheta }(t),t\geq 0$
as the process whose density is the Green function of the following
equation, for $\alpha \in (0,1]$, $\beta \in (0,2],$ $|\vartheta |=\min
\{\beta ,2-\beta \},$%
\begin{equation}
^{C}D_{t}^{\alpha }u(x,t)=\mathcal{D}_{x}^{\beta ,\vartheta }u(x,t),\qquad
x\in \mathbb{R},\;t\geq 0,  \label{tsf}
\end{equation}%
where $\mathcal{D}_{x}^{\beta ,\vartheta }$ is the Riesz-Feller fractional
derivative with Fourier transform
\begin{equation*}
\widehat{\mathcal{D}_{x}^{\beta ,\vartheta }u}(\kappa )=-\psi _{\beta
,\vartheta }(\kappa )\widehat{u}(\kappa ),\qquad \kappa \in \mathbb{R},
\end{equation*}%
and $\psi _{\beta ,\vartheta }(\kappa ):=|\kappa |^{\beta }e^{i\,sign(\kappa
)\vartheta \pi /2}$ (see \cite{MAI2}, for details).

We also recall the definition of a stable random variable $\mathcal{S}%
_{\beta }$ with stability index $\beta \in (0,2]$ and symmetry parameter $%
|\vartheta |=\min \{\beta ,2-\beta \}$, which is defined by the following
characteristic function%
\begin{equation*}
\mathbb{E}e^{i\kappa \mathcal{S}_{\beta }}=e^{-\psi _{\beta ,\vartheta
}(\kappa )}=e^{-|\kappa |^{\beta }e^{i\,sign(\kappa )\vartheta \pi /2}}.
\end{equation*}%
We will consider hereafter $\mathcal{S}_{\beta }$ in the symmetric case,
i.e. we assume that $\vartheta =0.$

We recall that a (centered) random variable $X$ is said to be "in the domain
of attraction of $\mathcal{S}_{\beta }$" (and we write $X\in DoA(\mathcal{S}%
_{\beta })$), if the following convergence in law (by the extended central
limit theorem) holds for the rescaled sum of independent copies $X_{i},$ $%
i=1,2,...,$%
\begin{equation}
a_{n}\sum_{i=1}^{n}X_{i}\Longrightarrow \mathcal{S}_{\beta },  \label{cc}
\end{equation}%
where $\left\{ a_{n}\right\} _{n\geq 1}$ is a sequence such that $%
\lim_{n\rightarrow +\infty }a_{n}=0.$

\begin{theorem}
\label{thm2} Let $N_{A}^{(c)}(t)$, $t\geq 0,$ be the renewal process with
(rescaled) $k$-th renewal time $T_{k}^{A,c}:=c^{-1}\sum_{j=1}^{k}Z_{A,j}$,
where $Z_{A,j}$ are i.i.d. random variables with density (\ref{tk}), for $%
c>0 $ and let $X_{i}^{(c)}$ be i.i.d. centered r.v.'s with density $%
f_{X^{(c)}}$, (with scale parameter $1/c$), such that $\widehat{f}%
_{X^{(c)}}(\kappa/c )\simeq 1-(|\kappa |/c)^{\beta }$, for $c\rightarrow
+\infty.\ $Then the following convergence of the one-dimensional
distribution holds, as $c\rightarrow +\infty ,$%
\begin{equation}
c^{-\alpha''/\beta }\sum_{i=1}^{N_{A}^{(c)}(t)}X_{i}^{(c)}\Longrightarrow
Y_{\alpha'',\beta }(t),\qquad t>0,  \label{cc2}
\end{equation}%
where $Y_{\alpha'',\beta }(t)$ is the space-time fractional diffusion
process, whose transition density satisfies equation (\ref{tsf}), with time-derivative of order  $%
\alpha''=lim_{t \rightarrow +\infty}\alpha(t),$ $\beta \in (0,2]$
and $\vartheta =0.$
\end{theorem}

\begin{proof}
The characteristic function of (\ref{cc2}) can be written, for any $t\geq 0,$
as%
\begin{equation*}
\mathbb{E}e^{i\kappa c^{-\alpha''/\beta
}\sum_{i=1}^{N_{A}^{(c)}(t)}X_{i}^{(c)}}=\sum_{n=0}^{\infty }p_{n}^{A,c}(t)%
\left[ \widehat{f}_{X^{(c)}}(\kappa c^{-\alpha''/\beta })\right]
^{n},
\end{equation*}%
where $p_{n}^{A,c}(t):=P\left( N_{A}^{(c)}(t)=n\right) ,$ $t\geq 0,$ $%
n=0,1,....$ We note that%
\begin{eqnarray*}
p_{n}^{A,c}(t) &=&P(T_{n}^{A,c}<t)-P(T_{n+1}^{A,c}<t) \\
&=&P\left( \sum_{j=1}^{n}Z_{A,j}<ct\right) -P\left(
\sum_{j=1}^{n+1}Z_{A,j}<ct\right) =p_{n}^{A}(ct),
\end{eqnarray*}%
so that, by (\ref{pk}), we have
\begin{equation*}
\int_{0}^{+\infty }e^{-st}p_{n}^{A,c}(t)dt=\frac{1}{c}\frac{\lambda
^{n}(s/c)^{\frac{s}{c}A(s/c)-1}}{\left( \lambda +(s/c)^{\frac{s}{c}%
A(s/c)}\right) ^{n+1}}
\end{equation*}%
and%
\begin{eqnarray*}
\mathcal{L}\left\{ \mathbb{E}e^{i\kappa c^{-\alpha''/\beta
}\sum_{i=1}^{N_{A}^{(c)}(t)}X_{i}^{(c)}};s\right\} &=&\frac{1}{c}\frac{%
(s/c)^{\frac{s}{c}A(s/c)-1}}{(s/c)^{\frac{s}{c}A(s/c)}+\lambda \lbrack 1-%
\widehat{f}_{X}(\kappa c^{-\alpha''/\beta })]} \\
&=&\frac{s^{\frac{s}{c}A(s/c)-1}}{s^{\frac{s}{c}A(s/c)}+\lambda \lbrack 1-%
\widehat{f}_{X}(\kappa c^{-\alpha''/\beta })]c^{\frac{s}{c}A(s/c)}}.
\end{eqnarray*}%
We observe that $lim_{r\rightarrow 0^{+}}srA(sr)=\alpha''$ and thus $%
lim_{c\rightarrow +\infty }s^{\frac{s}{c}A(s/c)}=s^{\alpha''},$ by (\ref%
{con}). Moreover, by assumption, $\widehat{f}_{X}(\kappa c^{-\alpha'' /\beta })\simeq 1-c^{-\alpha''}|\kappa |^{\beta }$, for $c\rightarrow
+\infty .$ As a consequence, we have
\begin{equation*}
\lim_{c\rightarrow +\infty }\mathcal{L}\left\{ \mathbb{E}e^{i\kappa
c^{-\alpha'' /\beta }\sum_{i=1}^{N_{A}^{(c)}(t)}X_{i}^{(c)}};s\right\} =%
\frac{s^{\alpha''}-1}{s^{\alpha''}+\lambda |\kappa |^{\beta }}
\end{equation*}%
and, inverting the LT by means of (\ref{ml}), we can write%
\begin{equation}
\lim_{c\rightarrow +\infty }\mathbb{E}e^{i\kappa c^{-\alpha''/\beta
}\sum_{i=1}^{N_{A}^{(c)}(t)}X_{i}^{(c)}}=E_{\alpha''}(-\lambda t^{\alpha''}|\kappa |^{\beta }),  \label{em}
\end{equation}%
for any fixed $t\geq 0$. Formula (\ref{em}) coincides with the Fourier
transform of the Green function of (\ref{tsf}) (see \cite{MAI2}, for
details).
\end{proof}

The previous result reduces, in the fixed order case, to Theorem
IV.2 in \cite{SCAL}, if $\alpha(t)=\alpha''$, for any $t$; thus we
can conclude that, in the limit, the influence of the initial
parameter $\alpha'$ vanishes.

Let us now denote by $\overset{M_{1}}{\Longrightarrow }$ the convergence in
the $M_{1}$ topology in the Skorokhod space $D([0,T))$, for $T>0$ (see \cite%
{WHI} and \cite{MEE2} for details on the convergence in the $M_{1}$
topology).

\begin{theorem}
Let $Y_{A}^{(c)}(t):=\sum_{i=1}^{N_{A}(ct)}X_{i}^{(c)}$, then under the
assumptions of Theorem \ref{thm2}%
\begin{equation*}
\left\{ c^{-\alpha''/\beta }Y_{A}^{(c)}(t)\right\} _{t\geq 0}\overset{%
M_{1}}{\Longrightarrow }\left\{ Y_{\alpha'',\beta }(t)\right\} _{t\geq
0},\qquad c\rightarrow +\infty ,
\end{equation*}%
on $D([0,+\infty )).$
\end{theorem}

\begin{proof}
We start by proving that, for the r.v.'s $X_{i}^{(c)},$ the convergence in (%
\ref{cc}) holds for $a_{c}=c^{-1/\beta }$ since%
\begin{equation}
\mathbb{E}e^{i\kappa c^{-1/\beta }\sum_{i=1}^{c}X_{i}^{(c)}}=\left( \widehat{%
f}_{X^{(c)}}\left( \kappa c^{-1/\beta }\right) \right) ^{c}\simeq \left(
1-|\kappa |^{\beta }c^{-1}\right) ^{c},\qquad c\rightarrow +\infty .
\label{el}
\end{equation}%
Thus $X^{(c)}\in DoA(\mathcal{S}_{\beta })$, $c\rightarrow +\infty
.$ Under the assumptions on $\alpha (\cdot )$ and $A(\cdot )$
given in Theorem \ref{thm1}, we can easily see that
$T_{\left\lfloor ct\right\rfloor }^{A}:=\sum_{j=1}^{\left\lfloor
ct\right\rfloor }Z_{A,j}$ behaves asymptotically, for
$c\rightarrow +\infty ,$ as in the special case (of the
fractional Poisson process) where $Z_{A}$ is distributed as $\mathcal{A}%
_{\alpha }(Z),$ where $\mathcal{A}_{\alpha }(t),$ $t\geq 0,$ is an
 $\alpha $%
-stable subordinator (with $\alpha =\alpha'')$ and $Z$ is an
independent,
exponential r.v. with parameter $\lambda .$ Indeed, since, by (\ref{con}), $%
\lim_{s\rightarrow 0^{+}}sA(s)=\alpha'',$ we can derive that
\begin{equation*}
\mathcal{L}\left\{ P\left( Z_{A}>t\right) ;s\right\} \sim \frac{s^{\alpha''-1}}{s^{\alpha''}+\lambda },\qquad s\rightarrow 0^{+},
\end{equation*}%
by considering (\ref{uu}). Thus the following convergence holds $%
\{T_{\left\lfloor ct\right\rfloor }^{A}\}_{t\geq 0}\overset{J_{1}}{%
\Rightarrow }\{\mathcal{A}_{\alpha''}(t)\}_{t\geq 0},$ as $c\rightarrow
+\infty ,$ in $D([0,+\infty ))$ (see \cite{MEE2}, p.100).

By the independence of $Z_{A,j}$ and $X_{j}^{(c)},$ for any $j=1,2...$ and
by the functional central limit theorem, we have that%
\begin{equation*}
\left\{ c^{-1/\beta }\sum_{j=1}^{[ct]}X_{j}^{(c)},c^{-\alpha''}N_{A}(ct)\right\} _{t\geq 0}\overset{J_{1}}{\Longrightarrow }\left\{
\mathcal{S}_{\beta }(t),\mathcal{L}_{\alpha''}(t)\right\} _{t\geq
0},\qquad c\rightarrow +\infty ,
\end{equation*}%
in the $J_{1}$ topology on $D([0,+\infty )).$ Therefore, by the above
mentioned Theorem 2.1 in \cite{meer1}, the following convergence holds%
\begin{equation*}
\left\{ c^{-\alpha''/\beta }Y_{A}^{(c)}(t)\right\} _{t\geq 0}\overset{%
M_{1}}{\Longrightarrow }\left\{ \mathcal{S}_{\beta }(\mathcal{L}_{\alpha''}(t))\right\} _{t\geq 0},\qquad c\rightarrow +\infty ,
\end{equation*}%
which gives the desired result, by considering the well-known equality in
distribution $\mathcal{S}_{\beta }(\mathcal{L}_{\alpha''}(t))\overset{d}{=%
}Y_{\alpha'',\beta }(t)$ (see \cite{MAI2}).
\end{proof}

\begin{remark}
As a special case of the previous result, when $\beta =2$ and $\lambda =1/2$%
, we obtain the convergence of the process $Y_{A}^{(c)},$ for $c\rightarrow
\infty,$ to the so-called generalized\ grey Brownian motion $\mathcal{B}%
_{\alpha }(t),t\geq 0,$ (with $\alpha =\alpha''),$ which can be defined
by means of its characteristic function $\mathbb{E}e^{i\kappa \mathcal{B}%
_{\alpha }(t)}=E_{\alpha }(- t^{\alpha }\kappa ^{2}/2)$ (see \cite{MUR} and
\cite{MUR2}).
\end{remark}

\end{document}